\documentclass[a4paper,12pt]{article}
\usepackage[top=2.5cm,bottom=2.5cm,left=2.5cm,right=2.5cm]{geometry}
\usepackage{cite, amsmath, amssymb}
\usepackage{amssymb}
\usepackage{graphicx}
\newenvironment{proof}{{\noindent \it Proof.}}{\hfill $\blacksquare$\par}
\usepackage{latexsym}
\usepackage{graphicx,booktabs,multirow}
\usepackage{tikz}
\usetikzlibrary{decorations.pathreplacing}
\usetikzlibrary{intersections}
\usepackage{enumerate}
\usepackage{graphicx,booktabs,multirow}
\usepackage{appendix}
\newtheorem{theorem}{Theorem}[section]

\newtheorem{lemma}[theorem]{Lemma}

\begin{document}

\title{Extremal problems on Sombor indices of unicyclic graphs with a given diameter}
\author{Hechao Liu\thanks{Corresponding author}
 \\
{\small School of Mathematical Sciences, South China Normal University,}\\ {\small Guangzhou, 510631, P. R. China}\\
 \small {\tt hechaoliu@m.scnu.edu.cn}
}
\date{}
\maketitle
\begin{abstract}
Sombor index is a novel topological index, which was introduced by Gutman and defined for a graph $G$ as $SO(G)=\sum\limits_{uv\in E(G)}\sqrt{d_{u}^{2}+d_{v}^{2}}$, where $d_{u}=d_{G}(u)$ denotes the degree of vertex $u$ in graph $G$.

Extremal problems on the Sombor index for trees with a given diameter has been considered by
Chen et al. [H. Chen, W. Li, J. Wang, Extremal values on the Sombor index of trees, MATCH Commun. Math. Comput. Chem. 87 (2022) 23--49] and Li et al. [S. Li, Z. Wang, M. Zhang, On the extremal Sombor index of trees with a given diameter, Appl. Math. Comput. 416 (2022) 126731].
As an extension of results introduces above, we determine the maximum Sombor indices for unicyclic graphs with a fixed order and given diameter.
\end{abstract}


\hskip0.2cm

\maketitle

\makeatletter
\renewcommand\@makefnmark%
{\mbox{\textsuperscript{\normalfont\@thefnmark)}}}
\makeatother

 \baselineskip=0.30in

\section{Introduction}
\hskip 0.6cm
Topological indices are numerical invariants that can act as good predictors of the physicochemical properties of molecules. Consequently, some topological indices witnessed a wide range of applications in chemical sciences, pharmaceutical sciences, complex networks, toxicology and other fields.

Let $G=(V(G),E(G))$ be a connected graph with $|V(G)|=n$ and $|E(G)|=m$. Let $N_{G}(u)$ be the set of neighbor of vertex $u$, the $d_{G}(u)=|N_{G}(u)|$ denotes the degree of vertex $u$. If $d_{G}(u)=1$, then we call $u$ is a pendent vertex in $G$. Let $PV(G)$ the set of pendent vertices in $G$.
Let $d_{G}(u,v)$ be the distance between vertex $u$ and $v$ in $G$, then the diameter in $G$ is $d=\max\limits_{u,v \in V(G)} \{d_{G}(u,v)\}$.

Recently, a type of novel topological indices, (reduced) Sombor index, were introduced by Gutman, defined as \cite{gumn2021}
$$SO(G)=\sum_{uv\in E(G)}\sqrt{d_{u}^{2}+d_{v}^{2}}.$$
$$SO_{red}(G)=\sum_{uv\in E(G)}\sqrt{(d_{u}-1)^{2}+(d_{v}-1)^{2}}.$$
See\cite{aiva2021,chli2021,rirm2021,doai2021,dengt2021,fyli2021,gumn2021,guma2021,
hoxu2021,liwn2021,lizh2022,lich2021,lyhu2021,lyfh2021,tliu2021,redz2021,zylh2021} for more details about Sombor index.

Our paper is devoted to solve the extremal problem of unicyclic graphs with a given diameter, which is inspired by the recent paper \cite{chli2021}.
The paper \cite{chli2021} considered the Sombor indices of trees with given parameters, including matching number, pendent vertices, diameter, segment number, branching number, etc. Much has been written about the extremal value of topological indices of various graphs for a given diameter, for example, Li et al. \cite{lizh2022} determined the largest and the second largest Sombor indices of $n$-vertex trees with a given diameter $d\geq 4$. Jiang et al. \cite{jilu2021} determined the maximum augmented Zagreb index of trees with a given diameter. Zhong \cite{zhol2018} presented the minimum harmonic index for unicyclic graphs with a given diameter and characterize the corresponding extremal graphs. The extremal general Randi\'{c} index and general sum-connectivity index of unicyclic graphs with a given diameter was determined by Alfuraidan et al. \cite{advs2021,advs2022}. Other results see \cite{jhll2020,suik2019,ylih2013}. Motivated by \cite{zhol2018,advs2021,advs2022}, we determine the maximum Sombor indices for unicyclic graphs with a given diameter.


\section{Preliminaries}

Let $\mathcal{U}_{n,d}=\{G\ |\ G \ is\ a \ unicyclic\ graph\ with\ order\ n\ and\ diameter\ d \}$.
$U_{n}^{d}\in \mathcal{U}_{n,d}$ ($d\geq 4$) are shown in Figure 1.

\begin{figure}[ht!]
  \centering
  \scalebox{.18}[.18]{\includegraphics{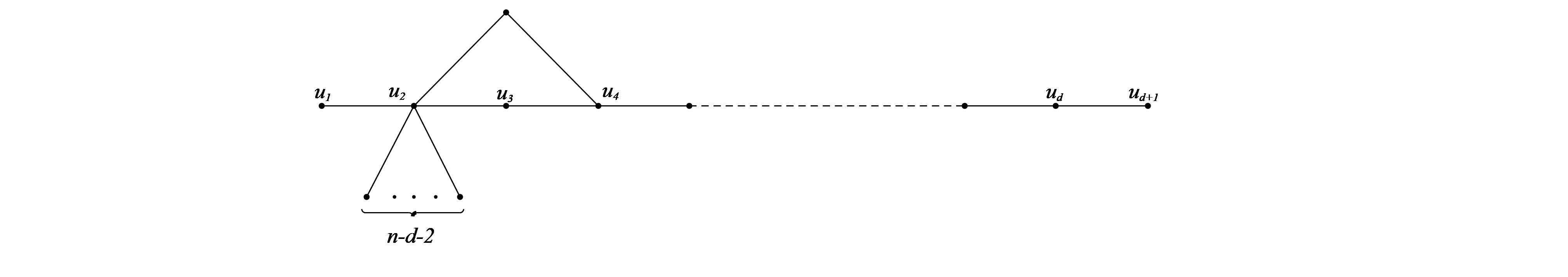}}
  \caption{The graph $U_{n}^{d}$ ($d\geq 4$).}
 \label{fig-1}
\end{figure}

For $4\leq d\leq n-2$, we have
$$SO(U_{n}^{d})=(n-d-1)\sqrt{(n-d+1)^{2}+1}+2\sqrt{(n-d+1)^{2}+4}+F_{1},$$
where $F_{1}=2\sqrt{13}+\sqrt{10}$ if $d=4$; $F_{1}=2\sqrt{2}(d-5)+3\sqrt{13}+\sqrt{5}$ if $d\geq 5$.

Similarly, we have
$$SO_{red}(U_{n}^{d})=(n-d-1)(n-d)+2\sqrt{(n-d)^{2}+1}+F_{2},$$
where $F_{2}=2\sqrt{5}+2$ if $d=4$; $F_{2}=\sqrt{2}(d-5)+3\sqrt{5}+1$ if $d\geq 5$.

In this paper, we will obtain the following results.

\begin{theorem}\label{t2-1}
Let $G\in \mathcal{U}(n,d)$, $4\leq d\leq n-2$. Then $SO(G)\leq SO(U_{n}^{d})$, with equality if and only if $G\cong U_{n}^{d}$, where $U_{n}^{d}$ is shown in Figure \ref{fig-1}.
\end{theorem}

\begin{theorem}\label{t2-2}
Let $G\in \mathcal{U}(n,d)$, $4\leq d\leq n-2$. Then $SO_{red}(G)\leq SO_{red}(U_{n}^{d})$, with equality if and only if $G\cong U_{n}^{d}$.
\end{theorem}


\section{Main results}

\begin{lemma}\label{l3-1}\cite{chli2021}
Let $\phi(x,y)=\sqrt{x^{2}+y^{2}}-\sqrt{(x-1)^{2}+y^{2}}$, where $x>1$ and $y>0$. Then $\phi(x,y)$ is a strictly monotonically increasing with $x$ and strictly monotonically decreasing with $y$.
\end{lemma}

\begin{lemma}\label{l3-01}\cite{lyfh2021}
Let $a(\geq 0)$, $p(\geq 1)$ be constants and $f(x)=x^{p}+(a-x)^{p}$. Then
$f(x)$ is monotonically decreasing when $x\leq \frac{a}{2}$, and $f(x)$ is monotonically increasing when $x\geq \frac{a}{2}$.
\end{lemma}

\begin{lemma}\label{l3-2}\cite{ctra2021}
Let $G\in \mathcal{U}_{n}$ $($$n\geq 5$$)$. Then
$SO(G)\geq 2\sqrt{2}n=SO(C_{n}),$
with equality if and only if $G\cong C_{n}$.
\end{lemma}

Throughout this article, we always suppose that $C=v_{1}v_{2}\cdots v_{|V(C)|}$ be the only one cycle and $P=u_{1}u_{2}\cdots u_{d}u_{d+1}$ be the diametral path of the unicyclic graph we considered. Let $PV(G)$ be the set of pendent vertices of $G$.

Denote by $\mathcal{U}_{n,d}^{max}$ the graph in $\mathcal{U}_{n,d}$ with maximum Sombor index.
By Lemma \ref{l3-2}, $\mathcal{U}_{n,d}^{max}$ must have at least one pendent vertex.
We also know that $n\geq d+2$, we start by considering $n=d+2$.

\begin{theorem}\label{t3-3}
Let $G\in \mathcal{U}_{d+2,d}$ $($$d\geq 4$$)$. Then
$SO(G)\leq SO(U_{d+2}^{d}),$
with equality if and only if $G\cong U_{d+2}^{d}$.
\end{theorem}
\begin{proof}
Let $G^{*}\in \mathcal{U}_{d+2,d}^{max}$.
Since $n=d+2$, there exists one vertex, say $w$, out of $P=u_{1}u_{2}\cdots u_{d}u_{d+1}$. Then $C=u_{i}u_{i+1}wu_{i}$ or $C=u_{i}u_{i+1}u_{i+2}wu_{i}$.

\noindent {\bf Claim 1}. $u_{1}\cap V(C)=\varnothing$ and $u_{d+1}\cap V(C)=\varnothing$.

Suppose $u_{1}\cap V(C)\neq \varnothing$, then $C=u_{1}u_{2}wu_{1}$ or $C=u_{1}u_{2}u_{3}wu_{1}$.

\noindent {\bf Case 1}. $C=u_{1}u_{2}wu_{1}$

Let $G^{**}=G^{*}-\{u_{1}w\}+\{u_{3}w\}$, then $G^{**}\in \mathcal{U}_{d+2,d}$, and
$$SO(G^{*})-SO(G^{**})=2\sqrt{2^{2}+2^{2}}-\sqrt{1^{2}+3^{2}}-\sqrt{3^{2}+3^{2}}=\sqrt{2}-\sqrt{10}<0.$$

\noindent {\bf Case 2}. $C=u_{1}u_{2}u_{3}wu_{1}$

Let $G^{**}=G^{*}-\{u_{1}w\}+\{u_{2}w\}$, then $G^{**}\in \mathcal{U}_{d+2,d}$, and
$$SO(G^{*})-SO(G^{**})=\sqrt{2}-\sqrt{10}<0,$$
contradict with the maximality of $G^{*}$.
Thus $u_{1}\cap V(C)=\varnothing$. Similarly, we also have $u_{d+1}\cap V(C)=\varnothing$.

\noindent {\bf Claim 2}. $|E(C)|=4$

If $|E(C)|\geq 5$, then contradiction with the choice of $P=u_{1}u_{2}\cdots u_{d}u_{d+1}$. Thus $|E(C)|=3$ or $|E(C)|=4$.
If $|E(C)|=3$, then $C=u_{i}u_{i+1}wu_{i}$. By Claim 1, we also know $2\leq i\leq d-1$, and $u_{i-1}\notin PV(G^{*})$ or $u_{i+2}\notin PV(G^{*})$ (Since $d\geq 4$).
We suppose $u_{i-1}\notin PV(G^{*})$, then $d_{G^{*}}(u_{i-2})=1$ or $2$.
Let $G^{**}=G^{*}-\{u_{i}w\}+\{u_{i-1}w\}$, then $G^{**}\in \mathcal{U}_{d+2,d}$, we have
\begin{equation}
\begin{split}
SO(G^{*})-SO(G^{**})&=\sqrt{d_{G^{*}}^{2}(u_{i-2})+2^{2}}-\sqrt{d_{G^{*}}^{2}(u_{i-2})+3^{2}}+\sqrt{3^{2}+3^{2}}-\sqrt{2^{2}+3^{2}}< 0.  \nonumber
\end{split}
\end{equation}
Thus, $|E(C)|=4$.

By Claim 1 and Claim 2, we know that $2\leq i\leq d-2$. If $i\neq 2$ and $i\neq d-2$, we have
\begin{equation}
\begin{split}
SO(G^{*})-SO(U_{d+2}^{d})&=  \sqrt{2^{2}+3^{2}}+\sqrt{1^{2}+2^{2}}-\sqrt{1^{2}+3^{2}}-\sqrt{2^{2}+2^{2}} \\
&=  \sqrt{5}+\sqrt{13}-\sqrt{10}-\sqrt{8}< 0.  \nonumber
\end{split}
\end{equation}

Thus $i= 2$ or $i= d-2$, i.e., $G^{*}\cong U_{d+2}^{d}$.
\end{proof}

\ \notag\

Next, we consider $n\geq d+3$.

\begin{lemma}\label{l3-4}
Let $G^{*}\in \mathcal{U}_{n,d}^{max}$ $($$3\leq d\leq n-3$$)$. If $u\in PV(G^{*})$ and $u\sim u_{2}$ or $u\sim u_{d}$, where $u_{2}, u_{d}\in P=u_{1}u_{2}\cdots u_{d}u_{d+1}$, then $|V(C)\cap V(P)|\geq 2$.
\end{lemma}
\begin{proof}
On the contrary, we suppose $|V(C)\cap V(P)|\leq 1$.

\noindent {\bf Case 1}. $|V(C)\cap V(P)|=0$

Since $|V(C)\cap V(P)|=0$, there exists a path $u_{i}z_{1}z_{2}\cdots z_{k}$ joining cycle $C$ and path $P$.
Then $3\leq i\leq d-1$ and $u_{i-1}\notin PV(G^{*})$, $u_{i+1}\notin PV(G^{*})$.

If $k\geq 2$, let $G^{**}=G^{*}-\{z_{1}z_{2}\}+\{u_{i}z_{2}\}$; if $k=1$, let $G^{**}$ be the graphs
obtained from $G^{*}$ by deleting edge $u_{i}z_{1}$, identifying $u_{i}$ and $z_{1}$, then adding a new
pendent edge to vertex $u_{i}$.
Then $G^{**}\in \mathcal{U}_{n,d}$.

\noindent {\bf Subcase 1.1}. $k=1$
\begin{equation}
\begin{split}
SO(G^{*})-SO(G^{**})&= \left(\sqrt{d_{G^{*}}^{2}(u_{i-1})+3^{2}}-\sqrt{d_{G^{*}}^{2}(u_{i-1})+5^{2}} \right)\\
&\quad  + \left(\sqrt{d_{G^{*}}^{2}(u_{i+1})+3^{2}}-\sqrt{d_{G^{*}}^{2}(u_{i+1})+5^{2}} \right)\\
&\quad  +\sqrt{3^{2}+3^{2}}+2\sqrt{3^{2}+2^{2}}-2\sqrt{2^{2}+5^{2}}-\sqrt{1^{2}+5^{2}} \\
&< 0.  \nonumber
\end{split}
\end{equation}

\noindent {\bf Subcase 1.2}. $k=2$
\begin{equation}
\begin{split}
SO(G^{*})-SO(G^{**})&= \left(\sqrt{d_{G^{*}}^{2}(u_{i-1})+3^{2}}-\sqrt{d_{G^{*}}^{2}(u_{i-1})+4^{2}} \right)\\
&\quad  + \left(\sqrt{d_{G^{*}}^{2}(u_{i+1})+3^{2}}-\sqrt{d_{G^{*}}^{2}(u_{i+1})+4^{2}} \right)\\
&\quad  +2\sqrt{2^{2}+3^{2}}-\sqrt{3^{2}+4^{2}}-\sqrt{1^{2}+4^{2}} \\
&< 0.  \nonumber
\end{split}
\end{equation}

\noindent {\bf Subcase 1.3}. $k\geq 3$
\begin{equation}
\begin{split}
SO(G^{*})-SO(G^{**})&= \left(\sqrt{d_{G^{*}}^{2}(u_{i-1})+3^{2}}-\sqrt{d_{G^{*}}^{2}(u_{i-1})+4^{2}} \right)\\
&\quad  + \left(\sqrt{d_{G^{*}}^{2}(u_{i+1})+3^{2}}-\sqrt{d_{G^{*}}^{2}(u_{i+1})+4^{2}} \right)\\
&\quad  +\sqrt{2^{2}+3^{2}}-\sqrt{2^{2}+4^{2}} \\
&< 0.  \nonumber
\end{split}
\end{equation}

\noindent {\bf Case 2}. $|V(C)\cap V(P)|=1$

Due to $|V(C)\cap V(P)|=1$, we may assume that $C=v_{1}v_{2}v_{3}\cdots v_{|C|}v_{1}$ and $u_{i}(v_{1})=V(C)\cap V(P)$.

\noindent {\bf Subcase 2.1}. $|V(C)|=3$

As $d>2$, then $d_{G^{*}}(u_{i-1})\geq 2$ or $d_{G^{*}}(u_{i+1})\geq 2$.
Without loss of generality, we let $d_{G^{*}}(u_{i+1})\geq 2$. We know that $d_{G^{*}}(u_{i})\geq 4$, $d_{G^{*}}(u_{i+2})\geq 1$.

Let $G^{**}=G^{*}-\{v_{2}v_{3}\}+\{v_{2}u_{i+1}\}$, then $G^{**}\in \mathcal{U}_{n,d}$. By Lemma \ref{l3-1} and $d_{G^{*}}(u_{i})\geq 4$, we have
\begin{equation}
\begin{split}
& SO(G^{*})-SO(G^{**})\\
& = \sqrt{2^{2}+2^{2}}-\sqrt{(d_{G^{*}}(u_{i+1})+1)^{2}+2^{2}}+ \sqrt{d_{G^{*}}^{2}(u_{i})+2^{2}}- \sqrt{d_{G^{*}}^{2}(u_{i})+1^{2}}\\
&\quad  +\sqrt{d_{G^{*}}^{2}(u_{i})+d_{G^{*}}^{2}(u_{i+1})}-\sqrt{d_{G^{*}}^{2}(u_{i})+(d_{G^{*}}(u_{i+1})+1)^{2}}\\
&\quad  +\sqrt{d_{G^{*}}^{2}(u_{i+1})+d_{G^{*}}^{2}(u_{i+2})}-\sqrt{(d_{G^{*}}(u_{i+1})+1)^{2}+ d_{G^{*}}^{2}(u_{i+2})}\\
&<     2\sqrt{2}-\sqrt{13}+\sqrt{d_{G^{*}}^{2}(u_{i})+2^{2}}- \sqrt{d_{G^{*}}^{2}(u_{i})+1^{2}} \\
&\leq   2\sqrt{2}-\sqrt{13}+2\sqrt{5}-\sqrt{17} \\
&< 0.  \nonumber
\end{split}
\end{equation}

\noindent {\bf Subcase 2.2}. $|V(C)|\geq 4$

Since $|V(C)|\geq 4$, we have $3\leq i\leq d-1$.
Let $C=v_{1}v_{2}\cdots v_{|V(C)|}$ $(v_{1}=u_{i})$.

Let $G^{**}=G^{*}-\{v_{2}v_{3}\}+\{u_{i}v_{3}\}$, then $G^{**}\in \mathcal{U}_{n,d}$.
\begin{equation}
\begin{split}
& SO(G^{*})-SO(G^{**})\\
& = \left( \sqrt{d_{G^{*}}^{2}(u_{i-1})+4^{2}}-\sqrt{d_{G^{*}}^{2}(u_{i-1})+5^{2}} \right)+ \left( \sqrt{d_{G^{*}}^{2}(u_{i+1})+4^{2}}-\sqrt{d_{G^{*}}^{2}(u_{i+1})+5^{2}} \right)\\
&\quad  +2\sqrt{2^{2}+4^{2}}-2\sqrt{2^{2}+5^{2}}+\sqrt{2^{2}+2^{2}}-\sqrt{1^{2}+5^{2}}\\
&< 0.  \nonumber
\end{split}
\end{equation}

Combine {\bf Case 1} and {\bf Case 2}, we know $|V(C)\cap V(P)|\geq 2$.
\end{proof}

\begin{lemma}\label{l3-5}
Let $G^{*}\in \mathcal{U}_{n,d}^{max}$ $($$3\leq d\leq n-3$$)$. There must exist a vertex $u_{0}\in VP(G^{*})$, we have $G^{*}-u_{0}\in \mathcal{U}_{n-1,d}$.
\end{lemma}
\begin{proof}
On the contrary, we suppose that $G^{*}-u\in \mathcal{U}_{n-1,d-1}$ for all vertices $u\in VP(G^{*})$.
By Lemma \ref{l3-2}, we know $VP(G^{*})\neq \varnothing$.

Let $P=u_{1}u_{2}\cdots u_{d}u_{d+1}$ be a diameter path of $G^{*}$ and $u_{1}\in VP(G^{*})$. Since $G^{*}-u\in \mathcal{U}_{n-1,d-1}$ for all vertices $u\in VP(G^{*})$, we have
$VP(G^{*})=\{u_{1},u_{d+1}\}$.
By Lemma \ref{l3-4}, $|V(C)\cap V(P)|\geq 2$.

Let $C=u_{i}u_{i+1}\cdots u_{i+l}v_{k}v_{k-1}\cdots v_{3}v_{2}v_{1}(u_{i}) $ $(l\geq 1)$.
It is obvious that $l\leq k$. Since $n\geq d+3$, then $k\geq 3$. $d_{G^{*}}(u_{i-1})\geq 1$, $d_{G^{*}}(u_{i})=3$, $d_{G^{*}}(v_{2})= d_{G^{*}}(v_{3})=2$.

\noindent {\bf Case 1}. $k=l$

Note that in this case, $k=l\geq 3$.
Let $G^{**}=G^{*}-\{u_{i}v_{2}\}-\{v_{2}v_{3}\}+ \{u_{i+1}v_{2}\}+ \{u_{i+1}v_{3}\}$.
\begin{equation}
\begin{split}
& SO(G^{*})-SO(G^{**})\\
& = \sqrt{d_{G^{*}}^{2}(u_{i-1})+3^{2}}-\sqrt{d_{G^{*}}^{2}(u_{i-1})+2^{2}}+2\sqrt{2^{2}+2^{2}}+2\sqrt{2^{2}+3^{2}}-3\sqrt{2^{2}+4^{2}}-\sqrt{1^{2}+4^{2}}. \nonumber
\end{split}
\end{equation}

\noindent {\bf Subcase 1.1}. $d_{G^{*}}(u_{i-1})=1$
$$SO(G^{*})-SO(G^{**})= \sqrt{10}-\sqrt{5}+4\sqrt{2}+2\sqrt{13}-6\sqrt{5}-\sqrt{17}<0.$$

\noindent {\bf Subcase 1.2}. $d_{G^{*}}(u_{i-1})\geq 2$
\begin{equation}
\begin{split}
& SO(G^{*})-SO(G^{**})\\
& \leq \sqrt{2^{2}+3^{2}}-\sqrt{2^{2}+2^{2}}+2\sqrt{2^{2}+2^{2}}+2\sqrt{2^{2}+3^{2}}-3\sqrt{2^{2}+4^{2}}-\sqrt{1^{2}+4^{2}}\\
& =2\sqrt{2}+3\sqrt{13}-6\sqrt{5}-\sqrt{17}<0. \nonumber
\end{split}
\end{equation}

\noindent {\bf Case 2}. $k>l$

Let $G^{**}=G^{*}-\{v_{2}v_{3}\}+\{u_{i}v_{3}\}$.
\begin{equation}
\begin{split}
& SO(G^{*})-SO(G^{**})\\
& = \sqrt{d_{G^{*}}^{2}(u_{i-1})+3^{2}}-\sqrt{d_{G^{*}}^{2}(u_{i-1})+4^{2}}+ \sqrt{d_{G^{*}}^{2}(u_{i+1})+3^{2}}-\sqrt{d_{G^{*}}^{2}(u_{i+1})+4^{2}}\\
&\quad +\sqrt{2^{2}+2^{2}}+\sqrt{2^{2}+3^{2}}-\sqrt{2^{2}+4^{2}}-\sqrt{1^{2}+4^{2}}\\
& < 2\sqrt{2}+\sqrt{13}-2\sqrt{5}-\sqrt{17}<0.\\ \nonumber
\end{split}
\end{equation}
This is a contradiction with $G^{*}\in \mathcal{U}_{n,d}^{max}$, thus the conclusion holds.
\end{proof}

\begin{lemma}\label{l3-6}
Let $G^{*}\in \mathcal{U}_{n,d}^{max}$ $($$4\leq d\leq n-3$$)$. Denote $\mathcal{U}^{*}=\{u\in VP(G^{*})| G^{*}-u\in \mathcal{U}_{n-1,d}\}$.
Let $v\in \bigcup\limits_{u\in \mathcal{U}^{*}}N_{G^{*}}(u)$, $Q_{G^{*}}(v)=\{w\in N_{G^{*}}(v)| d_{G^{*}}(w)\geq 2\}$.
Then there must exist a vertex $u_{0}\in VP(G^{*})$, $G^{*}-u_{0}\in \mathcal{U}_{n-1,d}$, and $|Q_{G^{*}}(N_{G^{*}}(u_{0}))|\geq 2$.
\end{lemma}
\begin{proof}
By Lemma \ref{l3-5}, $\mathcal{U}^{*}\neq \varnothing$.
We also know $|Q_{G^{*}}(v)|\geq 1$.

Suppose that for all $v\in \bigcup\limits_{u\in \mathcal{U}^{*}}N_{G^{*}}(u)$, we have $|Q_{G^{*}}(v)|= 1$.
Suppose that $C=v_{1}v_{2}\cdots v_{|V(C)|}$ be the only one cycle and $P=u_{1}u_{2}\cdots u_{d}u_{d+1}$ be the diametral path of $G^{*}$.

\noindent {\bf Claim 1}. $\mathcal{U}^{*}\subseteq N_{G^{*}}(u_{2})\bigcup N_{G^{*}}(u_{d})$

If there exists $u\in \mathcal{U}^{*}$, but $u\notin N_{G^{*}}(u_{2})\bigcup N_{G^{*}}(u_{d})$. Let
$N_{G^{*}}(u)=v$, then $v\notin \{u_{2}, u_{d}\}$.
We also know $v\notin \{u_{1}, u_{d+1}\}$.
As $|Q_{G^{*}}(v)|= 1$ for all $v\in \bigcup\limits_{u\in \mathcal{U}^{*}}N_{G^{*}}(u)$, then $v\notin V(P)\bigcup V(C)$.

Let $w\in N_{G^{*}}(v)$, $d_{G^{*}}(w)=t+1\geq 2$, $N_{G^{*}}(w)=\{v,x_{1},x_{2},\cdots, x_{t}\}$.

Let $G^{**}=G^{*}-\bigcup\limits_{1\leq i\leq t}\{x_{i}w\}+\bigcup\limits_{1\leq i\leq t}\{x_{i}v\}$, then $G^{**}\in \mathcal{U}_{n,d}$. By Lemma \ref{l3-01}, we have
\begin{equation}
\begin{split}
& SO(G^{*})-SO(G^{**})\\
& = \sum_{i=1}^{t}\left( \sqrt{(t+1)^{2}+d_{G^{*}}^{2}(x_{i})}- \sqrt{(d_{G^{*}}(v)+t)^{2}+d_{G^{*}}^{2}(x_{i})}   \right)\\
&\quad +(d_{G^{*}}(v)-1)\left( \sqrt{d_{G^{*}}^{2}(v)+1^{2}}- \sqrt{(d_{G^{*}}(v)+t)^{2}+1^{2}}   \right)\\
&\quad +\left( \sqrt{(t+1)^{2}+d_{G^{*}}^{2}(v)}- \sqrt{(d_{G^{*}}(v)+t)^{2}+1^{2}}   \right)\\
& < 0. \nonumber
\end{split}
\end{equation}
Thus, {\bf Claim 1} holds.

Since $\mathcal{U}^{*}=\{u\in VP(G^{*})| G^{*}-u\in \mathcal{U}_{n-1,d}\}\neq \varnothing$ and $\mathcal{U}^{*}\subseteq N_{G^{*}}(u_{2})\bigcup N_{G^{*}}(u_{d})$, we suppose that there exists $u\in \mathcal{U}^{*}$ and $u\in N_{G^{*}}(u_{2})$, then $|Q_{G^{*}}(u_{2})|=1$.
Thus $u_{1}\notin C$, $u_{2}\notin C$ (otherwise $|Q_{G^{*}}(u_{2})|\geq 2$), and by {\bf Claim 1}, we have $d_{G^{*}}(u_{2})\geq 3$ (otherwise $G^{*}-u_{1}\in \mathcal{U}_{n-1,d-1}$). Thus by Lemma \ref{l3-4}, $|V(C)\cap V(P)|\geq 2$.

Let $C=u_{i}u_{i+1}\cdots u_{i+l}v_{k}v_{k-1}\cdots v_{3}v_{2}v_{1}(u_{i})$ $(l\geq 1)$. It is obvious that $l\leq k$.

\noindent {\bf Claim 2}. $|V(C)\backslash V(P)|=1$

On the contrary, we suppose $|V(C)\backslash V(P)|\geq 2$, then $k\geq 3$.

\noindent {\bf Case 1}. $k=l$

Let $G^{**}=G^{*}-\{u_{i}v_{2}\}-\{v_{2}v_{3}\}+\{u_{i+1}v_{2}+\{u_{i+1}v_{3}\}$.
\begin{equation}
\begin{split}
& SO(G^{*})-SO(G^{**})\\
& = \left( \sqrt{d_{G^{*}}^{2}(u_{i+2})+2^{2}}- \sqrt{d_{G^{*}}^{2}(u_{i+2})+4^{2}}   \right)+ \left( \sqrt{d_{G^{*}}^{2}(u_{i-1})+3^{2}}- \sqrt{d_{G^{*}}^{2}(u_{i-1})+2^{2}}   \right)\\
&\quad +2\sqrt{2^{2}+3^{2}}+\sqrt{2^{2}+2^{2}}-2\sqrt{2^{2}+4^{2}}-\sqrt{1^{2}+4^{2}}\\
& < \sqrt{2^{2}+3^{2}}-\sqrt{2^{2}+2^{2}}  +2\sqrt{2^{2}+3^{2}}+\sqrt{2^{2}+2^{2}}-2\sqrt{2^{2}+4^{2}}-\sqrt{1^{2}+4^{2}} \\
& = 3\sqrt{13}-4\sqrt{5}-\sqrt{17}\\
& < 0. \nonumber
\end{split}
\end{equation}

\noindent {\bf Case 2}. $k>l$

Let $G^{**}=G^{*}-\{v_{2}v_{3}\}+\{u_{i}v_{3}\}$.
\begin{equation}
\begin{split}
& SO(G^{*})-SO(G^{**})\\
& = \left( \sqrt{d_{G^{*}}^{2}(u_{i-1})+3^{2}}- \sqrt{d_{G^{*}}^{2}(u_{i-1})+4^{2}}   \right)+ \left( \sqrt{d_{G^{*}}^{2}(u_{i+1})+3^{2}}- \sqrt{d_{G^{*}}^{2}(u_{i+1})+4^{2}}   \right)\\
&\quad +\sqrt{2^{2}+3^{2}}+\sqrt{2^{2}+2^{2}}-\sqrt{2^{2}+4^{2}}-\sqrt{1^{2}+4^{2}}\\
& < 0. \nonumber
\end{split}
\end{equation}
Thus, {\bf Claim 2} holds.

\noindent {\bf Claim 3}. $d_{G^{*}}(u_{d+1})=1$

On the contrary, we suppose $d_{G^{*}}(u_{d+1})=2$. Since $|Q_{G^{*}}(v)|= 1$ for all $v\in \bigcup\limits_{u\in \mathcal{U}^{*}}N_{G^{*}}(u)$.
Thus $u_{d}$ does not connect to a pendent vertex. As $d\geq 4$, then $d_{G^{*}}(u_{d-2})\geq 2$. Since $|V(C)\backslash V(P)|=1$ (by {\bf Claim 2}), then $|V(C)|=3$ or $4$.

\noindent {\bf Case 1}. $|V(C)|=3$

The only one cycle $C=u_{d}u_{d+1}wu_{d}$. Let $G^{**}=G^{*}-\{u_{d+1}w\}+\{u_{d-1}w\}$, then $G^{**}\in \mathcal{U}_{n,d}$.
\begin{equation}
\begin{split}
& SO(G^{*})-SO(G^{**})\\
& = \sqrt{d_{G^{*}}^{2}(u_{d-2})+2^{2}}- \sqrt{d_{G^{*}}^{2}(u_{d-2})+3^{2}}  \\
&\quad +2\sqrt{2^{2}+3^{2}}+\sqrt{2^{2}+2^{2}}-\sqrt{1^{2}+3^{2}}-\sqrt{2^{2}+3^{2}}-\sqrt{3^{2}+3^{2}}\\
& < \sqrt{13}+2\sqrt{2}-\sqrt{10}-3\sqrt{2}\\
& < 0. \nonumber
\end{split}
\end{equation}

\noindent {\bf Case 2}. $|V(C)|=4$

The only one cycle $C=u_{d-1}u_{d}u_{d+1}wu_{d-1}$. Let $G^{**}=G^{*}-\{u_{d+1}w\}+\{u_{d}w\}$, then $G^{**}\in \mathcal{U}_{n,d}$.
\begin{equation}
\begin{split}
& SO(G^{*})-SO(G^{**})\\
& = 2\sqrt{2^{2}+2^{2}}-\sqrt{3^{2}+3^{2}}-\sqrt{1^{2}+3^{2}}  \\
& = \sqrt{2}-\sqrt{10} < 0. \nonumber
\end{split}
\end{equation}
Thus, {\bf Claim 3} holds.

\noindent {\bf Claim 4}. $|V(C)|=4$

On the contrary, we suppose $|V(C)|=3$, the only one cycle $C=u_{i}u_{i+1}wu_{i}$ $(3\leq i\leq d-1)$.
Since $\mathcal{U}^{*}=\{u\in VP(G^{*})| G^{*}-u\in \mathcal{U}_{n-1,d}\}\neq \varnothing$, then $d_{G^{*}}(u_{2})\geq 3$.

\noindent {\bf Case 1}. $i=3$

Let $G^{**}=G^{*}-\{u_{3}w\}+\{u_{2}w\}$, then $G^{**}\in \mathcal{U}_{n,d}$.
\begin{equation}
\begin{split}
& SO(G^{*})-SO(G^{**})\\
& = (d_{G^{*}}(u_{2})-1)\left(\sqrt{d_{G^{*}}^{2}(u_{2})+1^{2}}- \sqrt{(d_{G^{*}}(u_{2})+1)^{2}+1^{2}} \right) \\
&\quad +\sqrt{d_{G^{*}}^{2}(u_{2})+3^{2}}-2\sqrt{(d_{G^{*}}(u_{2})+1)^{2}+2^{2}}+\sqrt{3^{2}+3^{2}}\\
& < 0. \nonumber
\end{split}
\end{equation}

\noindent {\bf Case 2}. $i\geq 4$

Let $G^{**}=G^{*}-\{u_{i}w\}-\{u_{i+1}w\}+\{u_{2}w\}+\{u_{3}w\}$, then $G^{**}\in \mathcal{U}_{n,d}$.
\begin{equation}
\begin{split}
& SO(G^{*})-SO(G^{**})\\
& = (d_{G^{*}}(u_{2})-1)\left(\sqrt{d_{G^{*}}^{2}(u_{2})+1^{2}}- \sqrt{(d_{G^{*}}(u_{2})+1)^{2}+1^{2}} \right) \\
&\quad +\sqrt{d_{G^{*}}^{2}(u_{2})+2^{2}}+\sqrt{d_{G^{*}}^{2}(u_{i+2})+3^{2}}+\sqrt{3^{2}+3^{2}}\\
&\quad -2\sqrt{(d_{G^{*}}(u_{2})+1)^{2}+2^{2}}-\sqrt{d_{G^{*}}^{2}(u_{i+2})+2^{2}}+\sqrt{2^{2}+3^{2}}-\sqrt{2^{2}+2^{2}}\\
& \leq 2(\sqrt{3^{2}+1^{2}}-\sqrt{4^{2}+1^{2}})+\sqrt{d_{G^{*}}^{2}(u_{2})+2^{2}}-2\sqrt{(d_{G^{*}}(u_{2})+1)^{2}+2^{2}}\\
&\quad +(\sqrt{1^{2}+3^{2}}-\sqrt{1^{2}+2^{2}})+\sqrt{2}+\sqrt{13}\\
& \leq \sqrt{2}+3\sqrt{10}+\sqrt{13}-2\sqrt{17}-5\sqrt{5}\\
& < 0. \nonumber
\end{split}
\end{equation}
Thus, {\bf Claim 4} holds, then the only one cycle $C=u_{i}u_{i+1}u_{i+2}wu_{i}$ $(3\leq i\leq d-2)$, $d_{G^{*}}(u_{d+1})=1$, $d_{G^{*}}(u_{2})\geq 3$ and $d_{G^{*}}(u_{i+3})\geq1$.

\noindent {\bf Case 1}. $i=3$
\begin{equation}
\begin{split}
& SO(G^{*})-SO(U_{n}^{d})\\
& = (d_{G^{*}}(u_{2})-1)\left(\sqrt{d_{G^{*}}^{2}(u_{2})+1^{2}}- \sqrt{(d_{G^{*}}(u_{2})+1)^{2}+1^{2}} \right) \\
&\quad +\sqrt{d_{G^{*}}^{2}(u_{2})+3^{2}}+\sqrt{d_{G^{*}}^{2}(u_{6})+3^{2}}-\sqrt{d_{G^{*}}^{2}(u_{6})+2^{2}}\\
&\quad -2\sqrt{(d_{G^{*}}(u_{2})+1)^{2}+2^{2}}+\sqrt{2^{2}+3^{2}}\\
& < \sqrt{13}-\sqrt{20}+\sqrt{10}-\sqrt{5}+2(\sqrt{10}-\sqrt{17})\\
& < 0. \nonumber
\end{split}
\end{equation}

\noindent {\bf Case 2}. $4\leq i\leq d-2$
\begin{equation}
\begin{split}
& SO(G^{*})-SO(U_{n}^{d})\\
& = (d_{G^{*}}(u_{2})-1)\left(\sqrt{d_{G^{*}}^{2}(u_{2})+1^{2}}- \sqrt{(d_{G^{*}}(u_{2})+1)^{2}+1^{2}} \right) \\
&\quad -2\sqrt{(d_{G^{*}}(u_{2})+1)^{2}+2^{2}}+\sqrt{d_{G^{*}}^{2}(u_{2})+2^{2}}+2\sqrt{2^{2}+3^{2}}-\sqrt{2^{2}+2^{2}}\\
&\quad +\sqrt{d_{G^{*}}^{2}(u_{i+3})+3^{2}}-\sqrt{d_{G^{*}}^{2}(u_{i+3})+2^{2}}\\
& < \sqrt{13}-2\sqrt{20}+2\sqrt{13}-2\sqrt{2}+\sqrt{10}-\sqrt{5}\\
& < 0. \nonumber
\end{split}
\end{equation}

This is a contradiction with the assumption $G^{*}\in \mathcal{U}_{n,d}^{max}$.
Thus there must exist a vertex $u_{0}\in VP(G^{*})$, $G^{*}-u_{0}\in \mathcal{U}_{n-1,d}$, and $|Q_{G^{*}}(N_{G^{*}}(u_{0}))|\geq 2$.
\end{proof}

In the following, we will obtain the main results.

\begin{theorem}\label{t3-7}
Let $G\in \mathcal{U}(n,d)$ $($$4\leq d\leq n-2$$)$. Then $SO(G)\leq SO(U_{n}^{d})$, with equality iff $G\cong U_{n}^{d}$, where $U_{n}^{d}$ is shown in Figure \ref{fig-1}.
\end{theorem}
\begin{proof}
If $n=d+2$, conclusion holds (Theorem \ref{t3-3}).
Suppose conclusion holds for $n-1$.

Let $G^{*}\in \mathcal{U}(n,d)$, $4\leq d\leq n-3$ with maximum $SO(G^{*})$. By Lemma \ref{l3-6}, there must exist $u\in VP(G^{*})$, $G^{*}-u\in \mathcal{U}_{n-1,d}$, and $v=N_{G^{*}}(u)$ connected to at least two non-pendent vertices, say $w_{1}$, $w_{2}$.
Denote $|N_{G^{*}}(v)|=k$, then $3\leq k\leq n-d+1$. $N_{G^{*}}(v)=\{u,v_{1},v_{2},\cdots,v_{k-1}\}$. Denote $d_{G^{*}}(v_{i})=k_{i}$ for $1\leq i\leq k-1$.

Let $G^{**}=G^{*}-u$, then $G^{**}\in \mathcal{U}_{n-1,d}$.
\begin{equation}
\begin{split}
SO(G^{*})&=SO(G^{**})+\sqrt{k^{2}+1^{2}}+\sum_{i=1}^{k-1} \left(\sqrt{k^{2}+k_{i}^{2}}- \sqrt{(k-1)^{2}+k_{i}^{2}}  \right)\\
&\leq SO(U_{n-1}^{d})+ \sqrt{k^{2}+1^{2}}+ 2(\sqrt{k^{2}+2^{2}}-\sqrt{(k-1)^{2}+2^{2}})  \\
&\quad  +(k-3)(\sqrt{k^{2}+1^{2}}-\sqrt{(k-1)^{2}+1^{2}})\\
&= (n-d-2)\sqrt{(n-d)^{2}+1}+2\sqrt{(n-d)^{2}+4}+F_{1}\\
&\quad  +2\sqrt{k^{2}+2^{2}}+(k-2)\sqrt{k^{2}+1^{2}}-2\sqrt{(k-1)^{2}+2^{2}}-(k-3)\sqrt{(k-1)^{2}+1^{2}}\\
&\leq (n-d-2)\sqrt{(n-d)^{2}+1}+2\sqrt{(n-d)^{2}+4}+F_{1}\\
&\quad  +\sqrt{(n-d+1)^{2}+1}+2(\sqrt{(n-d+1)^{2}+2^{2}}-\sqrt{(n-d)^{2}+2^{2}})\\
&\quad  +(n-d-2)(\sqrt{(n-d+1)^{2}+1}-\sqrt{(n-d)^{2}+1})\\
&= (n-d-1)\sqrt{(n-d+1)^{2}+1}+2\sqrt{(n-d+1)^{2}+4}+F_{1}\\
&= SO(U_{n}^{d}).  \nonumber
\end{split}
\end{equation}
Thus $SO(G)\leq SO(U_{n}^{d})$, with equality if and only if $G\cong U_{n}^{d}$.
\end{proof}

Similarly to the proof of Theorem \ref{t3-7}, we also have
\begin{theorem}\label{t3-8}
Let $G\in \mathcal{U}(n,d)$ $($$4\leq d\leq n-2$$)$. Then $SO_{red}(G)\leq SO_{red}(U_{n}^{d})$, with equality if and only if $G\cong U_{n}^{d}$.
\end{theorem}

\section{Conclusions}

Let $U(n,a,b,c)$, where $a\geq b\geq c\geq 0$ and $a+b+c=n-3$, be a unicyclic graph obtained from $C_{3}$ by
attaching $a$, $b$ and $c$ pendent vertices to every vertex of $C_{3}$.
It was proved in \cite{ctra2021} that $U(n,n-3,0,0)$ is the maximum unicyclic graph with respect to Sombor index.
It is easily to prove that $U(n,n-4,1,0)$ is the second maximum unicyclic graph with respect to Sombor index.
Since $\{C_{3}\}=\mathcal{U}_{n,1}$, $U(n,n-3,0,0)\in \mathcal{U}_{n,2}$ and $U(n,n-4,1,0)\in \mathcal{U}_{n,3}$, thus $\{C_{3}\}$, $U(n,n-3,0,0)$ and $U(n,n-4,1,0)$ are the extremal graph with
maximum Sombor index among $\mathcal{U}_{n,1}$, $\mathcal{U}_{n,2}$ and $\mathcal{U}_{n,3}$, respectively.
\begin{figure}[ht!]
  \centering
  \scalebox{.16}[.16]{\includegraphics{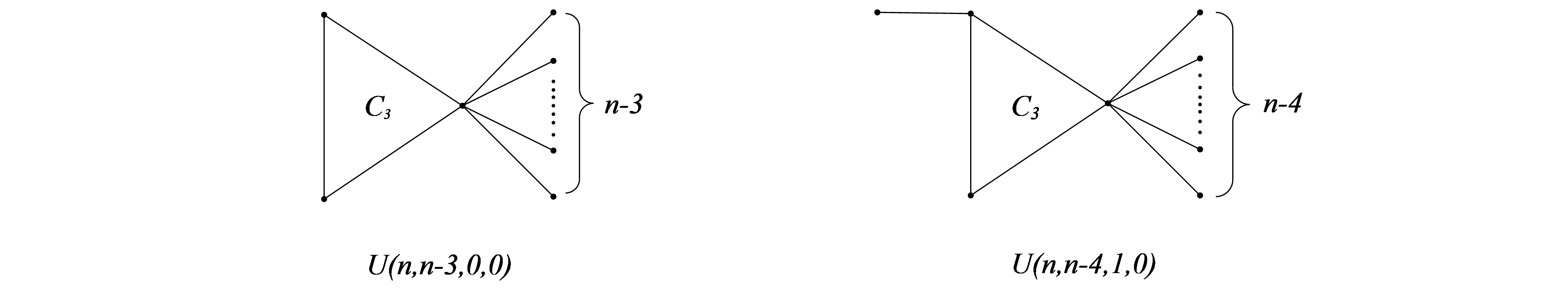}}
  \caption{The graphs $U(n,n-3,0,0)$ and $U(n,n-4,1,0)$.}
 \label{fig-2}
\end{figure}

Recently, Chen et al. \cite{chli2021} and Li et al. \cite{lizh2022} considered the Sombor indices of trees with a given diameter. In this paper, we completely determine the maximum Sombor indices for unicyclic graphs with a given diameter.
By the way, the maximum Sombor indices for bicyclic graphs with a given diameter had also been considered in our next paper.

\end{document}